\documentclass[11pt,a4paper]{amsart}
\usepackage{amsthm}
\usepackage{amsfonts}
\usepackage{amsmath}
\usepackage{amssymb}
\usepackage{dsfont}
\usepackage{txfonts}
\usepackage{cancel}
\usepackage{soul}
\usepackage{mathrsfs}
\usepackage{enumerate}
\usepackage{textcomp}
\usepackage{hyperref}
\usepackage{color}

\author[C. Goel, S. Kuhlmann, B. Reznick]{Charu Goel, Salma Kuhlmann, Bruce Reznick}
\address{Department of Mathematics and Statistics, University of
Konstanz, Universit\"{a}tsstrasse 10, 78457 Konstanz, Germany}
\email{charu.goel@uni-konstanz.de}

\address{Department of Mathematics and Statistics, University of
Konstanz, Universit\"{a}tsstrasse 10, 78457 Konstanz, Germany}
\email{salma.kuhlmann@uni-konstanz.de}

\address{Department of Mathematics, University of
Illinois at Urbana-Champaign, Urbana, IL 61801}
\email{reznick@math.uiuc.edu}

 \title[The analogue of Hilbert's 1888 theorem for Even Symmetric Forms]
{The analogue of Hilbert's 1888 theorem for Even Symmetric Forms\ }
\date{}


\numberwithin{equation}{section}

\begin{document}
\maketitle

\theoremstyle{definition}
  \newtheorem{thm}{Theorem}[section]
 \newtheorem{df}[thm]{Definition}
 \newtheorem{nt}[thm]{Notation}
 \newtheorem{prop}[thm]{Proposition}
 \newtheorem{rk}[thm]{Remark} 
 \newtheorem{lem}[thm]{Lemma}
 \newtheorem{obs}[thm]{Observation}
 \newtheorem{cor}[thm]{Corollary}
 \newtheorem{nthng}[thm]{}
\newcommand{\al}{\alpha}
\newcommand{\be}{\beta}
\newcommand{\ga}{\gamma}
\newcommand{\la}{\lambda}


\begin{abstract}
Hilbert proved in 1888 that a positive semidefinite (psd) real form is a sum of squares (sos) of real forms if and only if $n=2$ or $d=1$ or $(n,2d)=(3,4)$, where $n$ is the number of variables and $2d$ the degree of the form. We study the analogue for even symmetric forms. We establish that an even symmetric $n$-ary $2d$-ic psd form is sos if and only if $n=2$ or $d=1$ or $(n,2d)=(n,4)_{n \geq 3}$ or $(n,2d)= (3,8)$.
\end{abstract}

\section{Introduction}
A real form (homogeneous polynomial) $f$ is
called \textit{positive semidefinite} (psd) if it takes only non-negative values and it is called a \textit{sum of squares} (sos) if there exist other forms $h_j$ so that $f = h_1^2 +
\cdots + h_k^2$. Let $\mathcal{P}_{n,2d}$ and $\Sigma_{n,2d}$ denote the cone of psd and sos $n$-ary $2d$-ic forms (i.e. forms of degree $2d$ in $n$ variables)
respectively.

In 1888, Hilbert \cite{Hilb_1} gave a celebrated theorem that characterizes the pairs $(n,2d)$ for which every $n$-ary $2d$-ic psd form can be written as a  sos of forms. It states that every $n$-ary $2d$-ic psd form is sos if and only if  $n=2$ or $d=1$ or $(n,2d)=(3,4)$. Hilbert demonstrated that $\Sigma_{n,2d}\subsetneq \mathcal{P}_{n,2d}$ for $(n,2d)=(4,4), (3,6)$, thus reducing the problem to these two basic cases.

Almost ninety years later, Choi and Lam \cite{CL_1} returned to this subject. In particular, they considered the question of when a symmetric psd form is sos. A form $f(x_1,\ldots,x_n)$ is called \textit{symmetric} if $f(x_{\sigma(1)},\ldots,x_{\sigma(n)})= f(x_1,\ldots,x_n)$ for all $\sigma \in S_n$. 
As an analogue of Hilbert's approach, they reduced the problem to finding symmetric psd not sos $n$-ary $2d$-ics for the pairs $(n,4)_{n \geq 4}$ and $(3,6)$. They asserted the existence of psd not sos symmetric quartics in $n \geq 5$ variables; contingent on these examples, the answer is the same as that found by Hilbert. In \cite{G-K-R-1}, we constructed these quartic forms.

A form is \textit{even symmetric} if it is symmetric and in each of its terms every variable has even degree. 
\noindent Let $S\mathcal{P}^e_{n,2d}$ and $S\Sigma^e_{n,2d}$ denote the set of even symmetric psd and even symmetric sos $n$-ary $2d$-ic forms respectively. Set ${\Delta}_{n,2d}:=S\mathcal{P}^e_{n,2d} \setminus S\Sigma^e_{n,2d}$.
In this paper, we investigate the following question:
\vspace{-0.2cm}

\begin{equation} \label{Even symm ques} \mathcal{Q}(S^e): \text{For what pairs} \ (n,2d) \ \text{is} \ {\Delta}_{n,2d} = \emptyset \ ? 
\nonumber \end{equation}
\vspace{-0.2cm}

The current answers to this question in the literature are ${\Delta}_{n,2d} = \emptyset$ if $n=2, d=1, (n,2d)=(3,4)$ by Hilbert's Theorem, $(n,2d)=(3,8)$ due to Harris \cite{Har}, and $(n,2d)=(n,4)_{n \geq 4}$. The result ${\Delta}_{n,4} = \emptyset$ for $n \geq 4$ was attributed to Choi, Lam and Reznick in \cite{Har}; a proof can be found in \cite[Proposition 4.1]{Goel}. Further, ${\Delta}_{n,2d} \neq \emptyset$ for $(n,2d)=(n,6)_{n \geq 3}$ due to Choi, Lam and Reznick \cite{CLR-3}, for $(n,2d)=(3,10), (4,8)$ due to Harris \cite{Har-2} and for $(n,2d)=(3,6)$ due to Robinson  \cite{Rob-1}. Robinson's even symmetric psd not sos ternary sextic is the form
\[
R(x,y,z):= x^6+y^6+z^6-(x^4y^2+y^4z^2+z^4x^2+x^2y^4+y^2z^4+z^2x^4)+3x^2y^2z^2.
\]
Thus the answer to $\mathcal{Q}(S^e)$ in the literature can be summarized by the following chart:
\vspace{0.1cm}

\begin{center}
    \begin{tabular}{ | l | l | l | l | l | p{0.6cm} }
    \hline
    deg $\setminus$ var & 2 & 3 & 4 & 5 & \ $\ldots$ \\ \hline
    2 & $\checkmark$  & $\checkmark$ & $\checkmark$ & $\checkmark$ & \ $\ldots$ \\ \hline
    4& $\checkmark$ & $\checkmark$ & $\checkmark$ & $\checkmark$ & \ $\ldots$ \\ \hline
    6 & $\checkmark$ & $\times$& $\times$ &$\times$ & \ $\ldots$  \\ \hline
      8 & $\checkmark$ & $\checkmark$ & $\times$ & o  &  o \\ \hline
       10 & $\checkmark$ & $\times$ &o & o & o  \\ \hline
              12 & $\checkmark$ & o & o & o & o  \\ \hline
                14 & $\checkmark$ & o & o & o & o  \\ \hline
        $\vdots$ &  $\vdots$ & o & o & o & o \\
    \end{tabular}
\end{center}
\vspace{0.2cm}

 \noindent where, a tick ($\checkmark$)  denotes a positive answer to $\mathcal{Q}(S^e)$, a cross ($\times$) denotes a negative answer to $\mathcal{Q}(S^e)$, and a circle (o) denotes  \textquotedblleft undetermined\textquotedblright. Indeed to get a complete answer to $\mathcal{Q}(S^e)$, we need to investigate the question in these remaining cases, namely $(n,8)$ for $n \geq 5$,  $(3,2d)$ for $d \geq 6$ and $(n,2d)$ for $n \geq4, d \geq 5$.
\vspace{0.1cm}

\noindent \textbf{Main Theorem.} An even symmetric $n$-ary $2d$-ic psd form is sos if and only if $n=2$ or $d=1$ or $(n,2d)=(n,4)_{n \geq 3}$ or $(n,2d)= (3,8)$.
\vspace{0.1cm}

In other words, every \textquotedblleft o\textquotedblright \ in the chart can be replaced by \textquotedblleft $\times$\textquotedblright.
\vspace{0.1cm}

The article is structured as follows.
In Section \ref{sec:Hilbert for even symm}, we develop the tools (Theorem \ref{degJUMP} and Theorem \ref{reducedCases}) we need to prove our Main Theorem. These tools allow us to reduce to certain basic cases, in the same spirit as Hilbert and Choi-Lam. In Section \ref{sec:PNSsymmOctics} and Section \ref{sec:PNSsymmDecicsDodecics} we resolve those basic cases by producing explicit examples for $(n,2d); n \geq 4, d=4,5,6$. We conclude Section \ref{sec:PNSsymmDecicsDodecics} by interpreting even symmetric psd forms in terms of preorderings using our Main Theorem.
Finally, for ease of reference we summarize 
our examples 
in Section \ref{sec:Glossary}. 

\section{Reduction to basic cases}
\label{sec:Hilbert for even symm}

The following Lemma will be used in Theorem \ref{degJUMP}.

\begin{lem} \label{pn,qn irred} For $n \ge 3$, the even symmetric real forms
\[
\begin{gathered}
p_n(x_1,\dots,x_n) = 4 \sum_{j=1}^n x_j^4 - 17
 \sum_{1 \le i < j \le n} x_i^2x_j^2;  \\
q_n(x_1,\dots,x_n) = \sum_{j=1}^n x_j^6 + 3 \sum_{1 \le i \neq j \le n} x_i^4x_j^2 - 100 \sum_{1 \le i < j < k \le n} x_i^2x_j^2 x_k^2
\end{gathered}
\]
are irreducible over $\mathbb R$.
\end{lem}

\begin{proof} First observe that if a form $g$ has a factorization
\[
g(x_1,\dots,x_n) = \prod_{r=1}^u f_r(x_1,\dots,x_n),
\]
then the same holds when $x_{k+1} = \dots = x_n =
0$, hence it suffices to show that $p_3$ and $q_3$ are irreducible
over $\mathbb R$. Second, observe that if (in addition) $g$ is even and
symmetric, then for all $\sigma \in S_n$ and choices of sign,
$f_r(\pm_1 x_{\sigma_1},\dots,\pm_n x_{\sigma_n})$ is also a factor of $g$. We call distinct
(non-proportional) forms of this kind {\it cousins} of $f_r$. If (in
addition) $f_r$ is irreducible, $\deg f_r = d$ and  $\deg g = n$, then
$f_r$ can have at most $n/d$ cousins.

If $p_3$ is reducible, then it has a factor of degree $\le 2$.
Suppose that $p_3$ has a linear factor $\alpha_1 x_1 + \alpha_2 x_2 + \alpha_3
x_3$. Upon setting $x_3=0$, we see that
\[
\alpha_1 x_1 + \alpha_2 x_2\ |\ 4x_1^4 -17x_1^2 x_2^2  + 4x_2^4 =
(x_1+2x_2)(x_1-2x_2)(2x_1+x_2)(2x_1-x_2),
\]
so $\alpha_2/\alpha_1 \in \{\pm 1/2, \pm 2\}$. Similarly,
$\alpha_3/\alpha_2 \in \{\pm 1/2, \pm 2\}$, so $\alpha_3/\alpha_1 \in \{\pm
1/4, \pm 1, \pm 4\}$, which contradicts $\alpha_3/\alpha_1 \in \{\pm 1/2,
\pm 2\}$. It follows that $p_3$ has no linear factors.

Suppose $p_3$ has a quadratic (irreducible) factor $f = \alpha_1 x_1^2 + \alpha_2 x_2^2 + \alpha_3
x_3^2 + \dots$. If it is not true that $\alpha_1 = \alpha_2 = \alpha_3$, then
by permuting variables, $f$ has at least $3 > 4/2$ cousins. Thus
$\alpha_1 = \alpha_2 = \alpha_3$, and by scaling we may assume the common value
is 2. The binary quartic $4x_1^4 -17x_1^2 x_2^2 + 4x_2^4$ has six quadratic
factors, found by taking pairs of linear factors as above. Of these, the ones
in which $\alpha_1=\alpha_2$ are $2x_1^2 \pm 5 x_1x_2 + 2x_2^2$. It follows
that, more generally, the coefficient of $x_ix_j$ is $\pm 5$ and
that
\[
f(x_1,x_2,x_3) = 2x_1^2 + 2x_2^2 + 2x_3^2 + \pm_{12} (5 x_1x_2) +
\pm_{13} (5 x_1x_3) + \pm_{23} (5 x_2x_3).
\]
Regardless of the initial choice of signs, making the single sign
changes $x_i \mapsto -x_i$ for $i=1,2,3$ shows that $f$ has 4 cousins,
which again is too many. Therefore, we may conclude that $p_3$ is
irreducible.

We turn to $q_3$ and first observe that
\[
q_3(x_1,x_2,x_3) = (x_1^2 + x_2^2 + x_3^2)^3 -
106 x_1^2x_2^2x_3^2.
\]
Suppose now that $q_3$ is reducible, so it has at least one factor of degree $\le 3$, and let $f$ be such a factor of $q_3$.
Once again, we set $x_3 = 0$ and observe that
\[
f(x_1,x_2,0)\ |\ q_3(x_1,x_2,0) = (x_1^2 + x_2^2)^3.
\]
Since $x_1^2 + x_2^2$ is irreducible over $\mathbb R$, we conclude
that $\deg f = 2$ and $f(x_1,x_2) = \lambda(x_1^2 + x_2^2)$. Writing
\[
f(x_1,x_2,x_3) =  \alpha_1 x_1^2 + \alpha_2 x_2^2 + \alpha_3 x_3^2 +
\sum_{1  \le i < j \le 3} \beta_{ij}x_ix_j,
\]
we see from the foregoing that $\alpha_1 = \alpha_2$ and $\beta_{12} = 0$. By
setting $x_2=0$ and $x_1=0$ in turn, we see that the $\alpha_i$'s are
equal and $\beta_{ij} = 0$, so $f$ is a multiple of $x_1^2 + x_2^2 +
x_3^2$. But since $ x_1^2x_2^2x_3^2$ is not a multiple of  $x_1^2 + x_2^2 +
x_3^2$, $f$ cannot divide $q_3$, completing the proof.
\end{proof}

  \begin{lem}  \label{irred indef preserves non sosness} Let $f $  be a psd not sos $n$-ary $2d$-ic form and $p$ an irreducible indefinite 
form of degree $r$ in $\mathbb{R}[x_1, \ldots, x_n]$. Then the $n$-ary $(2d+2r)$-ic form $p^2f$ is also psd not sos.
\end{lem}
\begin{proof} See \cite[Lemma 2.1]{G-K-R-1}.
\end{proof}

\begin{thm}
\label{degJUMP} (\textbf{Degree Jumping Principle})
Suppose $f\in {\Delta}_{n,2d}$ for $n \geq 3$, then
\begin{enumerate}[{1.}]
\item for any integer $r \ge 2$, the form $f p_{n}^{2a} q_{n}^{2b} \in {\Delta}_{n,2d+4r}$, where $r=2a+3b$; $a, b \in \mathbb{Z}_{+}$, and $p_n, q_n$ are as defined in Lemma \ref{pn,qn irred}; 
\item $(x_{1} \ldots  x_n)^{2}f \in {\Delta}_{n,2d+2n}$. \end{enumerate}
\end{thm}

 \begin{proof}

 \begin{enumerate} [{1.}]
\item For $r  \in \mathbb{Z}_{+}, r  \geq 2$, there exists non-negative $a, b \in \mathbb{Z}$ such that $r=2a+3b$.
\noindent   Since $f p_{n}^{2a} q_{n}^{2b}$ is a product of even symmetric forms, it is even and symmetric; since it is a product of psd forms, it is psd. Thus we have $f p_{n}^{2a} q_{n}^{2b} \in S\mathcal{P}^{e}_{n,2d+4r}$.
\noindent Since $p_n$ and $q_n$ are indefinite and irreducible forms by Lemma \ref{pn,qn irred}, we get $fp_{n}^2 \in {\Delta}_{n,2d+8}$ and $fq_{n}^2 \in {\Delta}_{n,2d+12}$ by Lemma \ref{irred indef preserves non sosness}.
\noindent Finally, by repeating this argument we get $f p_{n}^{2a} q_{n}^{2b} \in \Delta_{n,2d+4r}$.  \vspace{0.1cm}

\item Taking $p$ $=x_{i}$ in turn for each $1 \leq i \leq n$, the assertion follows by Lemma  \ref{irred indef preserves non sosness}.
\end{enumerate}
\end{proof}

 \begin{thm}
\label{reducedCases} \textbf{(Reduction to Basic Cases)} If ${\Delta}_{n,2d} \neq \emptyset$ for $(n,8)_{n \geq 4}, (n,10)_{n \geq 3}$ and $(n,12)_{n \geq 3}$, then ${\Delta}_{n,2d} \neq \emptyset$ for $(n,2d)_{n \geq 3, d \geq 7}$.
\end{thm}

\begin{proof} For $n=3$, the basic examples are $R(x,y,z)  \in {\Delta}_{3,6}$ (by Robinson \cite{Rob-1}), several examples in ${\Delta}_{3,10}$ (by Harris \cite{Har}) and $p_{3}^{2} R(x,y,z) \in {\Delta}_{3,14}$ (by Theorem \ref{degJUMP} (1)). Every even integer $\geq 12$  can be written as $6+6k, 10+ 6k$ or $14+6k$, $k \geq 0$, and so by Theorem \ref{degJUMP} (2), ${\Delta}_{3,2d}$ is non-empty for $2d \geq 6$, $2d \neq 8$.

For $n \geq 4$, ${\Delta}_{n,6} \neq \emptyset$ (by Choi, Lam, Reznick \cite{CLR-3}). We shall show in Sections  \ref{sec:PNSsymmOctics} and  \ref{sec:PNSsymmDecicsDodecics} that ${\Delta}_{n,8}, {\Delta}_{n,10}, {\Delta}_{n,12}$ are non-empty. Every even integer $\geq 14$  can be written as $6+8k, 8+8k,  10+ 8k$ or $12+8k$ and so, given our claimed examples, by Theorem \ref{degJUMP}, ${\Delta}_{n,2d}$ is non-empty for $n \geq 4$, $2d \geq 6$.
\end{proof}

   In order to find psd not sos even symmetric $n$-ary octics, psd not sos even symmetric $n$-ary decics and psd not sos even symmetric $n$-ary dodecics for $n \geq 4$, we first recall the following theorems which will be particularly useful in proving the main results of Sections  \ref{sec:PNSsymmOctics} and \ref{sec:PNSsymmDecicsDodecics}. 

\begin{thm} \label{lem:SquaresinEvenFormReven} Suppose $p= \displaystyle \sum_{i=1}^{r} h^2_{i}$ is an even sos form. Then we may write $p=\displaystyle \sum_{j=1}^{s} q^2_{j}$, where each form $q^2_{j}$ is even. In particular, $q_{j}(\underline{x})= \displaystyle \sum c_j (\alpha) \underline{x}^{\alpha}$, where the sum is taken over $\alpha$'s in one congruence class mod 2 component-wise.
\end{thm}
\begin{proof} See \cite[Theorem 4.1]{CLR-3}.
\end{proof}

\begin{thm} \label{atmost 2 distinct comps} A symmetric $n$-ary quartic $f$ is psd if and only if $f(\underline{x}) \geq 0$ for every $\underline{x} \in \mathbb{R}^n$ with at most two distinct coordinates (if $n\geq4$). 
\end{thm}

\begin{proof} This was originally proved in \cite{CLR-2}; see \cite[Corollary 3.11]{Goel}, \cite[Section 2]{Har}.
\end{proof}

\begin{thm}
\label{thm:symmpsdnotsosquartics}
(i) For odd $2m+1 \geq 5$, the symmetric $2m+1$-ary quartic
\[
L_{2m+1}(\underline{x}):=\displaystyle m(m+1)\sum_{i<j}^{}(x_i-x_j)^4 - \bigg(\sum_{i<j}^{}(x_i-x_j)^2\bigg)^2
\]
 is psd not sos.

\noindent (ii) For $2m \geq 4$, the symmetric $2m$-ary quartic
\[
C_{2m}(x_1, \ldots, x_{2m}):= L_{2m+1}(x_1, \ldots, x_{2m}, 0)
\]
 is psd not sos.
\end{thm}
\begin{proof} See \cite[Theorems 2.8, 2.9]{G-K-R-1}.
\end{proof}

\begin{thm}
\label{CLR_quad} For an integer $r \geq 1$, let $M_r = M_r(x_1,\dots,x_n):= x_1^r + \cdots + x_n^r$. For reals $a, b, c$, the sextic $p=a M_2^3 +
bM_2M_4 + c M_6$ is psd if and only if $a t^2 + bt + c \ge 0$ for $t \in
\{1,2,\dots n\}$ and sos if and only if $a t^2 + bt + c \ge 0$ for $t \in \{1\} \cup [2,n]$.
\end{thm}
\begin{proof} See \cite[Theorems 3.7, 4.25]{CLR-3}.
\end{proof}

\begin{obs} Let $v_t$ denote any $n$-tuple with $t$ components equal to 1 and $n-t$ components equal to zero. Then $M_r(v_t)=t$, so $p(v_t)=t (a t^2 + bt + c)$. 
It will be useful in the proofs of Theorems 3.1, 4.1 and 4.4  to let $v_t(a_1, \ldots, a_t)$ denote the particular $v_t$ with $1$'s in positions $a_1, \ldots, a_t$.
\end{obs}


 \section{Psd not sos even symmetric $n$-ary octics for $n \geq 4$} 
 \label{sec:PNSsymmOctics}

It follows from Theorem \ref{thm:symmpsdnotsosquartics} that for $m \ge 2$,
\[
\begin{gathered}
G_{2m+1}(x_1,\dots,x_{2m+1}):= L_{2m+1}(x_1^2,\dots,x_{2m+1}^2) \in
S\mathcal{P}^e_{2m+1,8}; \\
D_{2m}(x_1,\dots,x_{2m}):= G_{2m+1}(x_1,\dots,x_{2m},0) \in S\mathcal{P}^e_{2m,8}.
\end{gathered}
\]
We showed in \cite{G-K-R-1} that $G_{2m+1}(\underline{x}) = 0$ for those $\underline{x} \in
\mathbb R^{2m+1}$ which are a permutation of $m+1$ $r$'s and  $m$
$s$'s for $(r,s) \in \mathbb R^2$, so that $D_{2m}(\underline{x}) = 0$,
projectively,  at any $v_m$ or $v_{m+1}$.

\begin{thm} \label{EvenSymmPSDnotSOSOctic} For $m \geq2$,
$D_{2m} \in \Delta_{2m,8}$ and $G_{2m+1} \in \Delta_{2m+1,8}$.
\end{thm}
\begin{proof} We observe that $D_{2m}(v_1) > 0$; in fact, it is equal to $m(m+1)(2m)-(2m)^2 = 2m^2(m-1)$. Thus, the coefficient of $x_i^8$ in $D_{2m}$ is
positive. Suppose $D_{2m} = \sum h_t^2$. Then $x_i^4$ must appear with
non-zero coefficient in at least one $h_t$. Since we may assume that
$h_t^2$ is even (using Theorem \ref{lem:SquaresinEvenFormReven}), we must have
\[
h_t = \sum_{i=1}^{n} a_i x_i^4 + \sum_{1 \le i < j \le n} b_{i,j} x_i^2 x_j^2.
\]
Since $D_{2m}(v_m)=  D_{2m}(v_{m+1}) = 0$, it follows that  $h_t(v_m)
= h_t(v_{m+1}) = 0$, and this holds for all permutations of $v_m$ and
$v_{m+1}$. Our goal is to show that these equations imply that $h_t =
0$, which will contradict the assumption that $D_{2m}$ is
sos. By symmetry, it suffices to prove that $a_i = 0$ for one choice
of $i$.

To this end,
let $y^{(1)}=v_m(1, \ldots, m-1, 2m-1), y^{(2)}=v_m(1, \ldots, m-1, 2m)$ and $y^{(3)}=v_{m+1}(1, \ldots, m-1, 2m-1,2m)$. Then
\[
\begin{gathered}
0 = h_t(y^{(1)}) = \sum_{i=1}^{m-1} a_i + a_{2m-1} +
\sum_{1 \le i<j \le  m-1} b_{i,j}  + \sum_{i=1}^{m-1} b_{i,2m-1}; \\
0 = h_t(y^{(2)}) = \sum_{i=1}^{m-1} a_i + a_{2m} +
\sum_{1 \le i<j \le  m-1} b_{i,j}  + \sum_{i=1}^{m-1} b_{i,2m}; \\
0 = h_t(y^{(3)}) = \sum_{i=1}^{m-1} a_i + a_{2m-1} + a_{2m} +
\sum_{1 \le i<j \le  m-1} b_{i,j}  + \sum_{i=1}^{m-1} b_{i,2m-1} +
\sum_{i=1}^{m-1} b_{i,2m} + b_{2m-1,2m}.
\end{gathered}
\]
Taking the first equation plus the second minus the third yields
\[
 \sum_{i=1}^{m-1} a_i + \sum_{1\le i < j \le m-1} b_{i,j} = b_{2m-1,2m}.
\]
Since $m \ge 2$, $m-1 < 2m-2$; thus, the same argument implies that
\[
\sum_{i=1}^{m-1} a_i + \sum_{1\le i < j \le m-1} b_{i,j} = b_{2m-2,2m}.
\]
That is, the coefficient of $x_{2m-1}^2x_{2m}^2$ in $h_t$ equals the
coefficient of  $x_{2m-2}^2x_{2m}^2$, and so by symmetry, for all
distinct $i,j,k,\ell$,  the
coefficient of $x_i^2x_j^2$ equals the coefficient of $x_i^2x_k^2$,
which equals the coefficient of $x_k^2x_{\ell}^2$. Thus, for all
$i\neq j$,
$b_{i,j} = u$ for some $u$.

Subtracting the first from the second
equation gives now $a_{2m-1} = a_{2m}$, and so for all $i$,  $a_i =
v$ for some $v$. Finally, our previous equations imply that
\[
\begin{gathered}
0 = mv + \binom m2 u = (m+1)v + \binom {m+1}2 u = 0  \\
\implies
-v = \frac{m-1}2 \cdot u =  \frac{m}2 \cdot u \implies u  = 0 \implies
 v = 0.
\end{gathered}
\]
In other words,  $h_t = 0$, establishing the contradiction.

\vspace{0.1cm}

Suppose now that $G_{2m+1}$ were sos. Then
\[
G_{2m+1} = \sum_{t=1}^{r} h_t^2 \implies D_{2m} = \sum_{t=1}^{r} h_t^2(x_1,\dots,x_{2m},0),
\]
a contradiction. Thus $G_{2m+1}$ is not sos.
\end{proof}

\vspace{0.1cm}

\begin{rk} It was asserted in \cite{CLR-3} that the psd even symmetric $n$-ary octic
\[M_2\Big(M_2^3-(2k+1)M_2M_4 + k(k+1)M_6\Big) 
\]
  is not sos, provided $2 \le k \le n-2$. We prove this below for $k=2$ and $n \ge
4$.
\end{rk}

\begin{thm} \label{NewEvenSymmPSDnotSOSOctic}
For $n \ge4$, 
\[
T_{n}(x_1,\dots,x_n)=M_2\Big(M_2^3-5M_2M_4 + 6M_6\Big) \in \Delta_{n,8}.
\]
\end{thm}
\begin{proof} Note that $T_n$ is psd by Theorem 2.8. Suppose
\[
T_{n}(x_1,\dots,x_n) =  \sum_{r=1}^m h_r^2(x_1,\dots,x_n).
\]
Then, $T_{n}(v_2) = T_{n}(v_3) = 0$ but $T_{n}(v_1) > 0$. In particular, the terms $x_j^4$ must appear on the
right hand side. As in the proof of Theorem \ref{EvenSymmPSDnotSOSOctic}, these
terms must appear in
\[
\sum_{k=1}^n a_k x_k^4 + \sum_{1 \le j < k \le n} b_{jk} x_j^2x_k^2,
\]
which must vanish at every $v_2$ and every $v_3$. In particular, for
$i<j<k$, we have
\[
\begin{gathered}
a_i + a_j + b_{ij} = 0, \\
a_i + a_k + b_{ik} = 0, \\
a_j + a_k + b_{jk} = 0, \\
a_i + a_j + a_k + b_{ij} + b_{ik} + b_{jk} = 0.
\end{gathered}
\]
It easily follows that $a_i + a_j + a_k = 0$. Now assume
$i,j,k$ are distinct, but not necessarily increasing. Since $n \ge
4$, there is an unused index $\ell$ and we may conclude that  $a_i +
a_j + a_\ell = 0$. Hence $a_k = a_{\ell}$. Since these are arbitrary,
we conclude that $a_m$ is independent of $m$, and since $a_i + a_j +
a_k = 0$, it follows that each $a_m = 0$, a contradiction.
\end{proof}

\begin{rk} For $n = 3$, $M_2(M_2^3 - 5M_2M_4+ 6M_6) = 2M_2 R$ \ is
sos, see \cite{Rob-1}, or equation (7.4) in \cite{CLR-3}.
\end{rk}

\section{Psd not sos even symmetric $n$-ary decics and dodecics for $n \geq 4$}
 \label{sec:PNSsymmDecicsDodecics}

 \begin{thm} \label{EvenSymmPSDnotSOSDecic}
 For $n \ge 4$,
\[
P_n(x_1, \ldots, x_n)=(nM_4 - M_2^2)(M_2^3 - 5M_2M_4 + 6M_6) \in \Delta_{n,10}.
\]
 \end{thm}

 \begin{proof} First recall that
\[
nM_4 - M_2^2 = n\sum_{k=1}^n x_k^4 - \left(\sum_{k=1}^n x_k^2\right)^2
= \sum_{i < j} (x_i^2 - x_j^2)^2
\]
is psd by Cauchy-Schwarz. The zero set is $(\pm1, \dots, \pm 1)$.

Second, recall from Theorem \ref{CLR_quad} 
that the quadratic $t(a t^2 + bt + c)$ gives the value of the sextic $a M_2^3 + b M_2M_4 + c M_6$ at an $n$-tuple
$v_t$ with $t$ 1's and $n-t$ 0's. Since $t(t-2)(t-3) \ge 0$, this criterion is
satisfied, and the second factor is also psd with zeros at $v_2$ and
$v_3$. 

It follows that $P_n$ is psd and its coefficient of
$x_1^{10}$ is $(n-1)(1-5+6) > 0$. We show that $P_n$ is not sos
by showing that in any sos expression, $x_1^{10}$ cannot occur.

Using Theorem \ref{lem:SquaresinEvenFormReven}, we see that if $P_n = \sum h_r^2$
and $x_1^5$ occurs in $h_r$, then
\[
h_r = a x_1^5 + x_1^3\left(\sum_{k=2}^n b_k x_k^2 \right) +
 x_1\left(\sum_{k=2}^n c_k x_k^4 \right) +
x_1\left(\sum_{2 \le j <  k < n} d_{jk} x_j^2x_k^2 \right).
\]
Since $P_{n}(v_2(1,j)) = P_n(v_3(1,j,k)) = 0$ for all $j,k$, $2 \le j < k \le n$, it follows that $0 = h_r(v_2(1,j)) = h_r(v_3(1,j,k))=0$, and we have the equations
\[
\begin{gathered}
0 = a + b_j + c_j, \\
0 = a + b_j + b_k + c_j + c_k + d_{jk} = (a + b_j + c_j) + (a + b_k + c_k)
+ d_{jk} - a.
\end{gathered}
\]
From these equations, we may conclude that for all  $2 \le j < k \le n$,
\[
b_k + c_k = -a, \qquad d_{jk} = a.
\]
Finally, $P_n(v_n) = 0$, so $h_r(v_n) = 0$; that is,
\[
0 = a + \sum_{k=2}^n (b_k + c_k) + \sum_{2 \le j <  k < n} d_{jk} =
a\left(1 - (n-1) + \binom{n-1}2\right) = a\cdot \frac{(n-2)(n-3)}2.
\]
Thus, $a = 0$ and $x_1^5$ occurs in no $h_r$.  This
gives the contradiction.
 \end{proof}

 \begin{rk} When $n = 3$, $P_n$ is sos:

\noindent $P_3=(3M_4 - M_2^2)(M_2^3 - 5M_2M_4 + 6M_6)$ 

$=4(x^4 + y^4 + z^4 - x^2y^2 -
x^2z^2 - y^2z^2)R(x,y,z)$ 

$= 4(x^2(x^2-y^2)^2(x^2-z^2)^2 + y^2(y^2 - x^2)^2 (y^2 - z^2)^2 + z^2
(z^2 - x^2)^2 (z^2 - y^2)^2).$
 \end{rk}

 \begin{rk} We have also shown that for $m \geq 2$, $\displaystyle  M_2 G_{2m+1} \in \Delta_{2m+1,10}$. We shall discuss $ M_2 G_{2m+1}$ and  $M_2 D_{2m}$ in a future publication.
 \end{rk}

 \begin{thm} \label{2ndEvenSymmPSDnotSOSDodecic}
 For $n \ge 5$,
\[
Q_n(x_1, \ldots, x_n) = (M_2^3 - 5M_2M_4 + 6M_6)(M_2^3 - 7M_2M_4+12M_6) \in \Delta_{n,12}.
\]
 \end{thm}
 \begin{proof}
 Since $(t-2)(t-3) \ge 0$ and
$(t-3)(t-4) \ge 0$, both factors in $Q_n$ are psd by Theorem \ref{CLR_quad}. The first has zeros at every
$v_2$ and $v_3$ and the second has zeros at every $v_3$ and $v_4$. But
note that neither has a zero at $v_1$. In fact, the coefficient of
$x_1^6$ in $Q_n$ is $(1-5+6)(1-7+12) > 0$. 

Suppose $Q_n$ is sos and $Q_n = \sum f_k^2$. As before, assume the
$f_k^2$'s are even (using Theorem \ref{lem:SquaresinEvenFormReven}).  Then $f_k(v_t) = 0$ for
every $v_t$ with $t = 2,3,4$. Since $Q_n(v_1) > 0$, there must be an $f_k$
containing $x_i^6$, which will be itself even. To this end, suppose
\[
f_k = \sum_{i=1}^n \al_i x_i^6 + \sum_{1 \leq i\neq j \leq n}\be_{ij}x_i^4x_j^2 +
\sum_{1 \leq i<j<k \leq n} \ga_{ijk}x_i^2x_j^2x_k^2.
\]
For $i < j$, let $\mu_{ij} = \be_{ij} + \be_{ji}$.
By evaluating at $v_2(i,j)$, we see that
\[
0 = \al_i + \al_j + \be_{ij} + \be_{ji} = \al_i + \al_j + \mu_{ij}
\implies \mu_{ij} = -\al_i -\al_j.
\]
By evaluating at $v_3(i,j,k)$, we have
\[
\begin{gathered}
0 = \al_i + \al_j + \al_k + \mu_{ij} + \mu_{ik} + \mu_{jk} + \ga_{ijk} =  \\
(\al_i+\al_j + \al_k) - 2(\al_i+\al_j + \al_k) + \ga_{ijk} \implies
\ga_{ijk} = (\al_i + \al_j + \al_k).
\end{gathered}
\]
Finally, by evaluating at $v_4(i,j,k,\ell)$, we have
\[
\begin{gathered}
0 = \al_i + \al_j + \al_k + \al_{\ell} + \mu_{ij} + \mu_{ik} +
\mu_{jk} +\mu_{i\ell} + \mu_{j\ell} + \mu_{k\ell} +  \ga_{ijk} +
\ga_{ij\ell}+\ga_{ik\ell} + \ga_{jk\ell} \\
= (\al_i + \al_j + \al_k + \al_{\ell})(1 -3+3) \implies \al_i + \al_j
+ \al_k + \al_{\ell} = 0.
\end{gathered}
\]
In other words, the sum of any four distinct $\al_r$'s is 0.
Since $n \ge 5$, there exists $m \in \{1,\dots,n\}$ different from
$i,j,k,\ell$ and
we have $\al_i + \al_j + \al_k + \al_{m} = 0$. Thus $\al_{\ell} =
\al_m$, and since the choice of $\ell$ and $m$ was arbitrary, we
conclude that $\al_1 = \dots = \al_n = \al$, so that $4\al = 0$ and
thus the coefficient of $x_i^6$ in $f_k$ must be zero, a
contradiction.
 \end{proof}

\begin{rk} We have been unable to determine whether $Q_3$ and $Q_4$ are sos.
\end{rk}

  \begin{thm}  \label{EvenSymmPSDnotSOSDodecic}
  For $n \ge 3$,
\begin{equation}
\begin{gathered}
R_n(x_1, \ldots, x_n) = \frac 1{12}\cdot (M_2^3 - 3M_2M_4+2M_6)(M_2^3 - 5M_2M_4 + 6M_6) \\
= \left( \sum_{1 \leq i<j<k \leq n} x_i^2x_j^2x_k^2\right)\left( \sum_{i=1}^{n} x_i^6 - \sum_{1 \leq i\neq j \leq n} x_i^4x_j^2 + 3\sum_{1 \leq i<j<k \leq n} x_i^2x_j^2x_k^2\right) \in \Delta_{n,12}.
\end{gathered}
\nonumber \end{equation}
 \end{thm}

 \begin{proof}
Since $(t-1)(t-2) \ge 0$ and $(t-2)(t-3) \ge 0$, 
both factors in $R_n$ are psd by Theorem \ref{CLR_quad}.
Moreover, the first factor implies that
\begin{equation} \label{eq:tu}
R_n(t,u,0,\dots,0) = 0
\end{equation}
for all real $t,u$, and at all $n$-tuples which are permutations of $(t,u,0,\dots,0)$.  We also have, for all $t$,
\begin{equation} \label{eq:pn:t11}
R_n(t,1,1,0,\dots,0) = t^2 \cdot \tfrac 12 \cdot ((2 + t^2)^3 -5(2+t^2)(2+t^4)+6(2+t^6)) = t^4(t^2-1)^2;
\end{equation}
this also holds by symmetry at any permutation of $n-3$ 0's, two 1's and one $t$.

We first remark that if $n=3$, $R_n(x,y,z) = x^2y^2z^2R(x,y,z)$; since $R$ is not sos, the same holds for $R$ multiplied by a product of squared linear factors. For $n \ge 4$,  more work is necessary.

Suppose $R_n = \sum_r h_r^2$, so that $\deg h_r = 6$. Suppose as usual that each $h_r^2$ is even (using Theorem \ref{lem:SquaresinEvenFormReven}). It follows from equation (\ref{eq:tu}) that for any such $h_r$,
$h_r(t,u,0,\dots,0) = 0$, for all $(t,u)$.
If, say, the terms in $h_r$ involving only $x_1^{6-k}x_2^k$ are $\sum_{k=0}^6 a_kx_1^{6-k}x_2^k$, then $\sum a_k t^{6-k}u^k = 0$ for all $t,u$, which implies that $a_k=0$. Proceeding similarly for all pairs of variables, we see that no monomial involving one or two variables can appear in any $h_r$.

For equation (\ref{eq:pn:t11}), let $\phi_r(t) =  h_r(t,1,1,0,\dots,0)$. We have
\begin{equation} \nonumber
\sum_{r=1}^{w} \phi_r(t)^2 = t^4(t^2-1)^2.
\end{equation}
Evaluation at $t = 0,1,-1$ shows that $\phi_r(t) = t(t^2-1)\psi_r(t)$, so that
\begin{equation} \nonumber
\sum_{r=1}^{w}  \psi_r(t)^2 = t^2,
\end{equation}
which in turn implies that $\psi_r(t) = \la_r t$ for some real $\la_r$. To recapitulate, we have
\begin{equation} \label{eq:hr:t111}
h_r(t,1,1,0,\dots,0) = \la_rt^2(t^2-1),
\end{equation}
and similar equations hold for all permutations of the variables.

Since $x_1^2x_2^2x_3^8$ appears in $r_n$ with coefficient 1, it also appears in $\sum h_r^2$ with coefficient 1. Since no monomial occurs in any $h_r$ with only two variables, it follows that $x_1x_2x_3^4$ must appear in at least one $h_r$, and since $h_r^2$ is even, all terms in $h_r$ must be $x_1x_2$ times an even quartic monomial. Further, we already know that $x_1^5x_2, x_1^3x_2^3, x_1x_2^5$ do not occur. Thus
\begin{equation} \nonumber
h_r(x_1,\dots,x_n) = x_1x_2\left( \sum_{j=3}^n (a_j x_1^2 x_j^2 + b_j x_2^2x_j^2 + c_j x_j^4) \ + \sum_{3 \le j < k \le n} d_{jk}x_j^2x_k^2 \right).
\end{equation}
Thus,
\begin{equation} \nonumber
\begin{gathered}
h_r(t,1,1,0,\dots,0) = t (a_3 t^2 + b_3  + c_3), \\
h_r(1,t,1,0,\dots,0) = t (a_3  + b_3t^2  + c_3).
\end{gathered}
\end{equation}
In view of equation (\ref{eq:hr:t111}), both of these cubics must be identically zero, hence $a_3 = b_3+c_3 = b_3 = a_3+c_3 = 0$, and so, in particular, $c_3 = 0$. This means that $x_1x_2x_3^4$ does {\it not} appear in any $h_r$, establishing the contradiction.
  \end{proof}

\noindent \textbf{Proof of the Main Theorem.}

\noindent Combine Theorems \ref{reducedCases}, \ref{EvenSymmPSDnotSOSOctic}, \ref{EvenSymmPSDnotSOSDecic}, \ref{2ndEvenSymmPSDnotSOSDodecic} and \ref{EvenSymmPSDnotSOSDodecic}.
\hfill $\Box$

\vspace{0.2cm}

\noindent

We now present an application of the Main Theorem to the interpretation of even symmetric psd forms in terms of preorderings. We briefly recall the necessary background. Let $S = \{g_1, \ldots, g_s\} \subseteq \mathbb{R}[\underline{x}]$, and let
\vspace{-0.2cm}

\[
T_{S} := \left\{ \displaystyle \sum _{e_1,\ldots,e_s \in \{0,1\}} \sigma_{\underline{e}} \ g^{e_1}_1\ldots g^{e_s}_s \ | \ \sigma_{\underline{e}} \in \Sigma \mathbb{R}[\underline {x}]^2, \underline{e} = (e_1,\ldots,e_s) \right\}
\]
be the associated finitely generated quadratic preordering, and
\[
K_{S} := \{\underline{x} \in \mathbb{R}[\underline{x}] \ | \ g_1(\underline{x}) \geq 0, \ldots, g_s(\underline{x}) \geq 0 \}
\]
be the associated basic closed semi-algebraic set.

\vspace{0.1cm}

We recall the following result which follows from  \cite[Proposition 6.1]{Scheiderer-1}:

\begin{prop} \label{prop:Scheiderer result}
Let $S$ be a finite subset of $\mathbb{R}[\underline{x}]$, such that dim($K_S)\geq3$. Then there exists a $g \in \mathbb{R}[\underline{x}]$ s.t. $g \geq 0$ on $K_S$ but $ g \notin T_S$.
\end{prop}
 In the concluding Remark \ref{f to g}, we investigate when can the form $g$ of Proposition \ref{prop:Scheiderer result} be chosen to be symmetric. We set $S':=\{x_1,\ldots,x_n\}$ and $K_{S'}=\mathbb{R}_{+}^n$ (the positive orthant).  We need the following relation between the preordering $T_{S'}$ and even sos forms, as verified in \cite[Lemma 1]{Fr-Hor}:

\begin{lem} \label{Lem:FH} Let $g \in \mathbb{R}[\underline{x}]$. Then $g(x_1^2, \ldots, x_n^2)$ 
is sos if and only if $g \in T_{S'}$.
\end{lem}

\begin{rk} \label{f to g}
Let $f$ be an even symmetric $n$-ary form of degree $2d$, and $g$ be the $n$-ary symmetric form of degree $d$
such that $g(x_1^2, \ldots, x_n^2)=f(x_1, \ldots, x_n)$. Clearly, $f$ is psd if and only if 
$g$ is nonnegative on 
$\mathbb{R}_{+}^n$. Moreover, by Lemma \ref{Lem:FH}, $f \in \Sigma \mathbb{R}[\underline {x}]^2$ if and only if $g \in T_{S'}$. Applying our Main Theorem,
we get 
that for $n \geq 3$, $d \geq 3$ and $(n,2d) \neq (3,8)$, there exists a {\it symmetric} $n$-ary $d$-ic form $g$ 
such that $g \geq 0$ on $\mathbb{R}_{+}^n$ but $ g \notin T_{S'}$.
\end{rk}
 
 \section{A glossary of forms}
\label{sec:Glossary}

For easy reference, we list the examples discussed in this paper.
\[
\begin{gathered}
L_{2m+1}(x_1, \ldots, x_{2m+1})
:=\displaystyle m(m+1)\sum_{i<j}^{}(x_i-x_j)^4 - \bigg(\sum_{i<j}^{}(x_i-x_j)^2\bigg)^2, \quad (\text{Theorem} \  \ref{thm:symmpsdnotsosquartics}); \\
C_{2m}(x_1, \ldots, x_{2m}) = L_{2m+1}(x_1, \ldots, x_{2m}, 0), \quad (\text{Theorem} \  \ref{thm:symmpsdnotsosquartics}); \\
M_r(x_1,\dots,x_n)= x_1^r + \cdots + x_n^r, \quad (\text{Theorem} \  \ref{CLR_quad}); \\
G_{2m+1}(x_1,\dots,x_{2m+1})= L_{2m+1}(x_1^2,\dots,x_{2m+1}^2), \quad (\text{Section} \  \ref{sec:PNSsymmOctics}); \\
D_{2m}(x_1,\dots,x_{2m})= G_{2m+1}(x_1,\dots,x_{2m},0), \quad (\text{Section} \  \ref{sec:PNSsymmOctics}); \\
T_{n}(x_1,\dots,x_n)=M_2\Big(M_2^3-5M_2M_4 + 6M_6\Big), \quad (\text{Theorem} \ \ref{NewEvenSymmPSDnotSOSOctic}); \\
P_n(x_1, \ldots, x_n)=(nM_4 - M_2^2)(M_2^3 - 5M_2M_4 + 6M_6), \quad (\text{Theorem} \  \ref{EvenSymmPSDnotSOSDecic}); \\
Q_n(x_1, \ldots, x_n) = (M_2^3 - 5M_2M_4 + 6M_6)(M_2^3 - 7M_2M_4+12M_6), \quad (\text{Theorem} \  \ref{2ndEvenSymmPSDnotSOSDodecic}); \\
R_n(x_1, \ldots, x_n) = \frac 1{12}\cdot (M_2^3 - 3M_2M_4+2M_6)(M_2^3 - 5M_2M_4 + 6M_6), \quad (\text{Theorem} \ \ref{EvenSymmPSDnotSOSDodecic}).
\end{gathered}
\]

\section*{Acknowledgements}
This paper contains results from the Ph.D. Thesis of Charu Goel \cite{Goel}, supervised by Salma Kuhlmann.  The authors are thankful to the Zukunftskolleg of the University of Konstanz for a mentorship program that supported the visit of Bruce Reznick to Konstanz and the visit of Charu Goel to Urbana in the summer of 2015. The first author acknowledges support from the Gleichstellungsrat of the University of Konstanz via a Postdoc Br\"{u}ckenstipendium in 2015 and the third author is thankful for the support of Simons Collaboration Grant 280987.

\end{document}